\documentclass[11pt]{amsart}

\usepackage[T1]{fontenc}
\usepackage[latin1]{inputenc}
\usepackage[english]{babel}
\usepackage{amsmath,amssymb,amsthm} 
\usepackage[matrix,arrow]{xy}
\usepackage{enumitem}
\usepackage{mymacros}
\usepackage{hyperref}

\theoremstyle{definition}
\newtheorem{prob}[thm]{Problem}

\renewcommand{\MR}[1]{}

\title{Henkin measures for the Drury-Arveson space}

\author{Michael Hartz}
\address{Department of Mathematics, Washington University in St. Louis, One Brookings Drive,
St. Louis, MO 63130, USA}
\email{mphartz@wustl.edu}
\thanks{The author was partially supported by a Feodor Lynen Fellowship}

\keywords{Drury-Arveson space, Henkin measure, totally singular measure, totally null set}
\subjclass[2010]{Primary 46E22; Secondary 47A13}

\begin{document}

\begin{abstract}
  We exhibit Borel probability measures on the unit sphere in $\bC^d$ for $d \ge 2$
  which are Henkin for the multiplier algebra of the Drury-Arveson space,
  but not Henkin in the classical sense.
  This provides a negative answer to a conjecture of Clou\^atre and Davidson.
\end{abstract}

\maketitle

\section{Introduction}

Let $\bB_d$ denote the open unit ball in $\bC^d$ and let $A(\bB_d)$ be the ball algebra,
which is the algebra of all analytic functions on $\bB_d$ which extend to be continuous on $\ol{\bB_d}$.
A regular complex Borel measure $\mu$ on the unit sphere $\bS_d = \partial \bB_d$ is said to be \emph{Henkin}
if the functional
\begin{equation*}
  A(\bB_d) \to \bC, \quad f \mapsto \int_{\bS_d} f \, d \mu,
\end{equation*}
extends to a weak-$*$ continuous functional on $H^\infty(\bB_d)$, the algebra of all bounded analytic
functions on $\bB_d$. Equivalently, whenever $(f_n)$ is a sequence in $A(\bB_d)$ which is uniformly
bounded on $\bB_d$ and satisfies $\lim_{n \to \infty} f_n(z) = 0$ for all $z \in \bB_d$, then
\begin{equation*}
  \lim_{n \to \infty} \int_{\bS_d} f_n \, d \mu = 0.
\end{equation*}

Henkin measures play a prominent role in the description
of the dual space of $A(\bB_d)$ and of peak interpolation sets for the ball algebra,
see Chapter 9 and 10 of \cite{Rudin08} for background material.
Such measures
are completely characterized by a theorem of Henkin \cite{Henkin68} and Cole-Range \cite{CR72}. To state
the theorem, recall that a Borel probability measure
$\tau$ on $\bS_d$ is said to be a representing measure for the origin if
\begin{equation*}
  \int_{\bS_d} f \, d \tau = f(0)
\end{equation*}
for all $f \in A(\bB_d)$.

\begin{thm}[Henkin, Cole-Range]
  A regular complex Borel measure $\mu$ on $\bS_d$ is Henkin if and only if it is absolutely continuous with respect
  to some representing measure for the origin.
\end{thm}

If $d=1$, then the only representing measure for the origin is the normalized
Lebesgue measure on the unit circle, hence the Henkin measures on the unit circle are precisely those
measures which are absolutely continuous with respect to Lebesgue measure.

In addition to their importance in complex analysis, Henkin measures also play a role
in multivariable operator theory \cite{Eschmeier97}.
However, it has become clear over the years that for the purposes of multivariable operator theory,
the ``correct'' generalization of $H^\infty$, the algebra of bounded analytic functions on the unit disc,
to higher dimensions
is not $H^\infty(\bB_d)$, but the multiplier algebra of the Drury-Arveson space $H^2_d$.
This is the reproducing kernel Hilbert space on $\bB_d$ with reproducing kernel
\begin{equation*}
  K(z,w) = \frac{1}{1- \langle z,w \rangle}.
\end{equation*}
A theorem of Drury \cite{Drury78} shows that $H^2_d$ hosts a version of von Neumann's inequality
for commuting \emph{row contractions}, that is, tuples $T= (T_1,\ldots,T_d)$
of commuting operators on a Hilbert space $\cH$ such that
the row operator $[T_1,\ldots,T_d]: \cH^d \to \cH$ is a contraction.
The corresponding dilation theorem is due to M\"uller-Vasilescu \cite{MV93}
and Arveson \cite{Arveson98}.
The Drury-Arveson space is also known as symmetric Fock space \cite{Arveson98,DP98},
it plays a distinguished role in the theory
of Nevanlinna-Pick spaces \cite{AM00,AM00a} and is an object of interest in harmonic analysis \cite{CSW11,ARS08}.
An overview of the various features of this space can be found in \cite{Shalit13}.

In \cite{CD16}, Clou\^atre and Davidson generalize much of the classical theory of Henkin measures
to the Drury-Arveson space. Let $\cM_d$ denote the multiplier algebra of $H^2_d$ and let $\cA_d$ be
the norm closure of the polynomials in $\cM_d$.
In particular, functions in $\cA_d$ belong to $A(\bB_d)$.
Clou\^atre and Davidson define
a regular Borel measure $\mu$ on $\bS_d$ to be \emph{$\cA_d$-Henkin} if
the associated integration functional
\begin{equation*}
  \cA_d \to \bC, \quad f \mapsto \int_{\bS_d} f \, d \mu
\end{equation*}
extends to a weak-$*$ continuous functional on $\cM_d$ (see Subsection \ref{ss:DA} for the definition
of weak-$*$ topology). Equivalently,
whenever $(f_n)$ is a sequence in $\cA_d$ such that $||f_n||_{\cM_d} \le 1$ for all $n \in \bN$
and $\lim_{n \to \infty} f_n(z) = 0$ for all $z \in \bB_d$, then
\begin{equation*}
  \int_{\bS_d} f_n(z) \, d \mu = 0,
\end{equation*}
see \cite[Theorem 3.3]{CD16}. This notion, along with the complementary notion of $\cA_d$-totally
singular measures, is crucial in the study of the dual space of $\cA_d$ and of peak interpolation sets
for $\cA_d$ in \cite{CD16}.

Compelling evidence of the importance of $\cA_d$-Henkin measures in multivariable
operator theory can be found in \cite{CD16a}, where Clou\^atre and Davidson extend
the Sz.-Nagy-Foias $H^\infty$--functional calculus
to commuting row contractions.
Recall that every contraction $T$ on a Hilbert space can be written as $T = T_{cnu} \oplus U$,
where $U$ is a unitary operator and $T_{cnu}$ is completely non-unitary (i.e. has no unitary summand).
Sz.-Nagy and Foias showed that in the separable case, $T$
admits a weak-$*$ continuous $H^\infty$-functional calculus if and only if the
spectral measure
of $U$ is absolutely continuous with respect to Lebesgue measure on the unit circle, see
\cite{SFB+10} for a classical treatment.
Clou\^atre and Davidson obtain a complete generalization of this result.
The appropriate generalization of a unitary is a \emph{spherical unitary}, which is
a tuple of commuting normal operators whose joint spectrum
is contained in the unit sphere.
Every commuting row contraction admits a decomposition $T = T_{cnu} \oplus U$, where $U$ is a spherical
unitary and $T_{cnu}$ is completely non-unitary
(i.e. has no spherical unitary summand), see \cite[Theorem 4.1]{CD16a}. The following result is
then a combination of Lemma 3.1 and Theorem 4.3 of \cite{CD16a}.

\begin{thm}[Clou\^atre-Davidson]
  Let $T$ be a commuting row contraction acting on a separable Hilbert space
  with decomposition $T = T_{cnu} \oplus U$ as above.
  Then $T$ admits a weak-$*$ continuous $\cM_d$-functional calculus if and only if the spectral measure
  of $U$ is $\cA_d$-Henkin.
\end{thm}

This result shows that for the theory of commuting row contractions, $\cA_d$-Henkin measures
are a more suitable generalization of absolutely continuous measures on the unit circle
than classical Henkin measures.
Thus, a characterization of $\cA_d$-Henkin measures would be desirable.

Since the unit ball of $\cA_d$ is contained in the unit ball of $A(\bB_d)$, it is trivial that
every classical Henkin measure is also $\cA_d$-Henkin.
Clou\^atre and Davidson conjectured \cite[Conjecture 5.1]{CD16} that conversely, every $\cA_d$-Henkin
measure is also a classical Henkin measure, so that these two notions agree.
If true, the classical theory would apply to $\cA_d$-Henkin measures and in particular, the
Henkin and Cole-Range theorem would provide a characterization of $\cA_d$-Henkin measures.
They also formulate a conjecture for the complementary notion of totally singular measure, which
turns out to be equivalent to their conjecture on Henkin measures \cite[Theorem 5.2]{CD16}.
Note that the conjecture is vacuously true if $d= 1$, as $\cM_1 = H^\infty$.

The purpose of this note is to provide a counterexample to the conjecture of Clou\^atre and Davidson for $d \ge 2$.
To state the main result more precisely, we require one more definition. A compact set $K \subset \bS_d$
is said to be \emph{totally null} if it is null for every representing measure of the origin.
By the Henkin and Cole-Range theorem, a totally null set cannot support a non-zero
classical Henkin measure.

\begin{thm}
  \label{thm:henkin_measure}
  Let $d \ge 2$ be an integer. There exists a Borel probability measure $\mu$ on
  $\bS_d$ which is $\cA_d$-Henkin and whose support is totally null.
\end{thm}

In fact, every measure which is supported on a totally null set is totally singular (i.e. it
is singular with respect to every representing measure of the origin). The measure in Theorem \ref{thm:henkin_measure} therefore also serves at the same time a counterexample to the conjecture
of Clou\^atre and Davidson on totally singular measures, even without invoking \cite[Theorem 5.2]{CD16}.

It is not hard to see that if $\mu$ is a measure on $\bS_d$ which satisfies
the conclusion of Theorem \ref{thm:henkin_measure}, then so does the trivial
extension of $\mu$ to $\bS_{d'}$ for any $d' \ge d$ (see Lemma \ref{lem:extension}),
hence it suffices to prove Theorem \ref{thm:henkin_measure} for $d = 2$.
In fact, the construction of such a measure $\mu$ is easier in the case $d=4$, so
will consider that case first.

The remainder of this note is organized as follows. In Section \ref{sec:prelim}, we recall some of
the necessary background material. Section \ref{sec:d4} contains the construction of a measure
$\mu$ which satisfies the conclusion of Theorem \ref{thm:henkin_measure} in the case $d=4$. In Section
\ref{sec:d2}, we prove Theorem \ref{thm:henkin_measure} in general.

\section{Preliminaries}
\label{sec:prelim}

\subsection{The Drury-Arveson space}
\label{ss:DA}

As mentioned in the introduction, the Drury-Arveson space $H^2_d$ is the reproducing kernel Hilbert
space on $\bB_d$ with reproducing kernel
\begin{equation*}
  K(z,w) = \frac{1}{1 - \langle z,w \rangle}.
\end{equation*}
For background material on reproducing kernel Hilbert spaces, see \cite{PR16} and \cite{AM02}.
We will require a more concrete description of $H^2_d$.
Recall that if $\alpha  = (\alpha_1,\ldots,\alpha_d)
\in \bN^d$ is a multi-index and if $z = (z_1,\ldots,z_d) \in \bC^d$, one usually writes
\begin{equation*}
  z^\alpha = z_1^{\alpha_1} \ldots z_d^{\alpha_d}, \quad \alpha! = \alpha_1! \ldots \alpha_d!, \quad
  |\alpha| = \alpha_1 + \ldots + \alpha_d.
\end{equation*}
The monomials $z^\alpha$ form an orthogonal basis of $H^2_d$, and
\begin{equation*}
  ||z^\alpha||^2_{H^2_d} = \frac{\alpha!}{|\alpha|!}
\end{equation*}
for every multi-index $\alpha$,
see \cite[Lemma 3.8]{Arveson98}.

Let $(x_n)$ and $(y_n)$ be two sequences of positive numbers. We write
$x_n \simeq  y_n$ to mean that there exist $C_1,C_2 > 0$ such that
\begin{equation*}
  C_1 y_n \le x_n \le  C_2 y_n \quad \text{ for all } n \in \bN.
\end{equation*}
The following well-known result can be deduced
from Stirling's formula, see \cite[p.19]{Arveson98}.

\begin{lem}
  \label{lem:stirling}
  Let $d \in \bN$. Then
  \begin{equation*}
    ||(z_1 z_2 \ldots  z_d)^n||_{H^2_d}^2 \simeq d^{-n d} (n+1)^{(d-1)/2}
  \end{equation*}
  for all $n \in \bN$. \qed
\end{lem}

The multiplier algebra of $H^2_d$ is
\begin{equation*}
  \cM_d = \{ \varphi: \bB_d \to \bC : \varphi f \in H^2_d \text{ for all } f \in H^2_d \}.
\end{equation*}
Every $\varphi \in \cM_d$ gives rise to a bounded multiplication operator $M_\varphi$ on $H^2_d$, and
we set $||\varphi||_{\cM_d} = ||M_\varphi||$. Moreover,
we may identify $\cM_d$ with a unital subalgebra of $B(H^2_d)$, the algebra of bounded operators
on $H^2_d$. It is not hard to see that $\cM_d$ is WOT-closed, and hence weak-$*$ closed, inside of $B(H^2_d)$.
Thus, $\cM_d$ becomes a dual space in this way, and we endow it with the resulting weak-$*$ topology.
In particular, for every $f,g \in H^2_d$, the functional
\begin{equation*}
  \cM_d \to \bC, \quad \varphi \mapsto \langle M_\varphi f,g \rangle,
\end{equation*}
is weak-$*$ continuous. Moreover, it is well known and not hard to see that on bounded subsets of $\cM_d$,
the weak-$*$ topology coincides with the topology of pointwise convergence on $\bB_d$.

\subsection{Henkin measures and totally null sets}

Let $K \subset \bS_d$ be a compact set. A function $f \in A(\bB_d)$ is said to \emph{peak on $K$}
if $f = 1$ on $K$ and $|f(z)| < 1$ for all $z \in \ol{\bB_d} \setminus K$.
Recall that $K$ is said to be \emph{totally null} if it is null for every representing measure of the
origin. In particular, if $d=1$, then $K$ is totally null if and only if it is a Lebesgue null set.
We will make repeated use of the following characterization of totally null
sets, see \cite[Theorem 10.1.2]{Rudin08}.

\begin{thm}
  \label{thm:peak}
  A compact set $K \subset \bS_d$ is totally null if and only if there exists a function
  $f \in A(\bB_d)$ which peaks on $K$.
\end{thm}

If $d' \ge d$, then we may regard $\bS_d \subset \bS_{d'}$ in an obvious way. Thus,
every regular Borel measure $\mu$ on $\bS_d$ admits a trivial extension $\widehat \mu$ to $\bS_{d'}$ defined by
\begin{equation*}
  \widehat \mu(A) = \mu(A \cap \bS_d)
\end{equation*}
for Borel sets $A \subset \bS_{d'}$.
The following easy lemma shows that it suffices to prove Theorem \ref{thm:henkin_measure}
in the case $d=2$. 

\begin{lem}
  \label{lem:extension}
  Let $\mu$ be a Borel probability measure on $\bS_d$, let $d' \ge d$ and let $\widehat \mu$
  be the trivial extension of $\mu$ to $\bS_{d'}$.
  \begin{enumerate}[label=\normalfont{(\alph*)}]
    \item If $\mu$ is $\cA_{d}$-Henkin, then $\widehat \mu$ is $\cA_{d'}$-Henkin.
    \item If the support of $\mu$ is a totally null subset of $\bS_{d}$, then
      the support of $\widehat \mu$ is a totally null subset of $\bS_{d'}$.
  \end{enumerate}
\end{lem}

\begin{proof}
  (a)
  Let $P: \bC^{d'} \to \bC^d$ denote the orthogonal projection onto the first $d$ coordinates.
  It follows from the concrete description of the Drury-Arveson space at the beginning
  of Subsection \ref{ss:DA} that
  \begin{equation*}
    V: H^2_d \to H^2_{d'}, \quad f \mapsto f \circ P,
  \end{equation*}
  is an isometry. Moreover, $V^* M_\varphi V = M_{\varphi|_{\bB_d}}$ for every
  $\varphi \in \cM_{d'}$, so that
  \begin{equation*}
    \cM_{d'} \mapsto \cM_d, \quad \varphi \mapsto \varphi \big|_{\bB_d},
  \end{equation*}
  is weak-$*$-weak-$*$ continuous and maps $\cA_{d'}$ into $\cA_d$.
  
  Suppose now that $\mu$ is $\cA_d$-Henkin. Then there exists a weak-$*$ continuous
  functional $\Phi$ on $\cM_d$ which extends the integration functional given by $\mu$,
  thus
  \begin{equation*}
    \int_{\bS_{d'}} \varphi \, d \widehat \mu = \int_{\bS_{d}} \varphi \, d \mu
    = \Phi( \varphi \big|_{\bB_d})
  \end{equation*}
  for $\varphi \in \cA_{d'}$. Since the right-hand side defines a weak-$*$ continuous
  functional on $\cM_{d'}$, we see that $\widehat \mu$ is $\cA_{d'}$-Henkin.

  (b) We have to show that if $K \subset \bS_d$ is totally null,
  then $K$ is also totally null as a subset of $\bS_{d'}$. But this is immediate from Theorem \ref{thm:peak}
  and the observation that if $f \in A(\bB_d)$ peaks on $K$, then $f \circ P \in A(\bB_{d'})$ peaks
  on $K$ as well, where $P$ denotes the orthogonal projection from (a).
\end{proof}

\section{The case \texorpdfstring{$d=4$}{d=4}}
\label{sec:d4}

The goal of this section is to prove Theorem \ref{thm:henkin_measure} in the case $d=4$
(and hence for all $d \ge 4$ by Lemma \ref{lem:extension}).
To prepare and motivate the construction of the measure $\mu$,
we begin by considering analogues of
Henkin measures for more general reproducing kernel Hilbert spaces
on the unit disc. Suppose that $\cH$ is a reproducing kernel Hilbert space on the unit disc $\bD$
with reproducing kernel of the form
\begin{equation}
  \label{eqn:circular_kernel}
  K(z,w) = \sum_{n=0}^\infty a_n (z \ol{w})^n,
\end{equation}
where $a_0 = 1$ and $a_n > 0$ for all $n \in \bN$.
If $\sum_{n=0}^\infty a_n < \infty$, then the series above
converges uniformly on $\ol{\bD} \times \ol{\bD}$, and $\cH$ becomes a reproducing kernel
Hilbert space of continuous functions on $\ol{\bD}$ in this way.
In particular, evaluation at $1$
is a continuous functional on $\cH$ and hence a weak-$*$ continuous functional on $\Mult(\cH)$.
Indeed,
\begin{equation*}
  \varphi(1) = \langle M_\varphi 1 , K(\cdot,1) \rangle_{\cH}
\end{equation*}
for $\varphi \in \Mult(\cH)$.
Therefore, the Dirac measure $\delta_{1}$ induces
a weak-$*$ continuous functional on $\Mult(\cH)$, but it is not absolutely
continuous with respect to Lebesgue measure, and hence not Henkin. (In fact, every regular Borel measure on the unit circle induces a weak-$*$ continuous functional on $\Mult(\cH)$.)

The main idea of the construction is to embed a reproducing kernel Hilbert space
as in the preceding paragraph into $H^2_4$.

To find the desired space $\cH$ on the disc, recall that
by the inequality of arithmetic and geometric means,
\begin{equation*}
  \sup \{ |z_1 z_2 \ldots z_d|: z \in \ol{\bB_d} \} = d^{-d/2},
\end{equation*}
and the supremum is attained if and only if $|z_1| = \ldots = |z_d| = d^{-1/2}$.
Hence,
\begin{equation*}
  r: \ol{\bB_4} \to \ol{\bD}, \quad z \mapsto 16 z_1 z_2 z_3 z_4,
\end{equation*}
indeed takes values in $\ol{\bD}$, and it maps $\bB_4$ onto $\bD$.
For $n \in \bN$, let
\begin{equation*}
  a_n = ||r(z)^n||^{-2}_{H^2_4},
\end{equation*}
and let $\cH$ be the reproducing kernel Hilbert space on $\bD$
with reproducing kernel
\begin{equation*}
  K(z,w) = \sum_{n=0}^\infty a_n (z \ol{w})^n.
\end{equation*}

\begin{lem}
  \label{lem:4_isom}
  The map
\begin{equation*}
  \cH \to H^2_4, \quad f \mapsto f \circ r,
\end{equation*}
is an isometry, and
$\sum_{n=0}^\infty a_n < \infty$.
\end{lem}

\begin{proof}
  It is well known that for any space on $\bD$ with kernel as in Equation
  \eqref{eqn:fourier_decay}, the monomials $z^n$ form an orthogonal basis
  and $||z^n||^2 = 1/a_n$ for $n \in \bN$. Thus, with our choice of $(a_n)$ above,
  we have
  \begin{equation*}
    ||z^n||^2 = \frac{1}{a_n} = ||r(z)^n||^2_{H^2_4}.
  \end{equation*}
  Since the sequence $r(z)^n$ is an orthogonal sequence in $H^2_4$,
  it follows that $V$ is an isometry.
  Moreover, an application of Lemma \ref{lem:stirling} shows that
  \begin{equation*}
    ||r(z)^n||^2 = 4^{4 n} ||(z_1,\ldots,z_4)^n||^2
    \simeq (n+1)^{3/2},
  \end{equation*}
  so that $a_n \simeq (n+1)^{-3/2}$, and hence $\sum_{n=0}^\infty a_n < \infty$.
\end{proof}

Let
\begin{equation*}
  h: \bT^3 \to \bS_4, \quad
  (\zeta_1,\zeta_2,\zeta_3) \mapsto 
  1/2 (\zeta_1, \zeta_2, \zeta_3, \ol{\zeta_1 \zeta_2 \zeta_3})
\end{equation*}
and observe that the range of $h$ is contained in $r^{-1}(\{1\})$.
Let $\mu$ be the pushforward of the normalized Lebesgue measure $m$ on $\bT^3$ by $h$,
that is,
\begin{equation*}
  \mu(A) = m ( h^{-1} (A))
\end{equation*}
for a Borel subset $A$ of $\bS_4$. We will show that
$\mu$ satisfies
the conclusion of Theorem \ref{thm:henkin_measure}.

\begin{lem}
  The support of $\mu$ is totally null.
\end{lem}

\begin{proof}
  Let $X = r^{-1}(\{1\})$, which is compact,
  and define $f = \frac{1+r}{2}$. Then $f$ belongs to the unit ball of $A(\bB_4)$
  and peaks on $X$,
  hence $X$ is totally null by Theorem \ref{thm:peak}.
  Since $h(\bT^3) \subset X$, the support of $\mu$ is contained in $X$, so the
  support of $\mu$ is totally null as well.
\end{proof}

The following lemma finishes the proof of Theorem \ref{thm:henkin_measure} in the
case $d=4$.

\begin{lem}
  The measure $\mu$ is $\cA_4$-Henkin.
\end{lem}

\begin{proof}
  Let $\alpha \in \bN^4$ be a multi-index. Then
  \begin{equation*}
    \int_{\bS_4} z^\alpha \, d \mu
    = \int_{\bT^3} z^\alpha \circ h \, dm
    = 2^{-|\alpha|} \int_{\bT^3} \zeta_1^{\alpha_1 - \alpha_4}
    \zeta_2^{\alpha_2 - \alpha_4} \zeta_3^{\alpha_3 - \alpha_4} \, dm.
  \end{equation*}
  This integral is zero unless $\alpha_4 = \alpha_1= \alpha_2 = \alpha_3 = \colon k$,
  in which case it equals $2^{- 4k}$.

  Let $g = K(\cdot,1) \circ r$, where $K$ denotes the reproducing kernel of $\cH$. Then $g \in H^2_4$
  by Lemma \ref{lem:4_isom}, and it is a power series in $z_1 z_2 z_3 z_4$. Thus,
  $z^\alpha$ is orthogonal to $g$ unless $\alpha_1 = \ldots = \alpha_4 = \colon k$, in which
  case
  \begin{equation*}
    \langle z^\alpha, g \rangle_{H^2_4} = 
    2^{-4 k} \langle r(z)^k, g \rangle_{H^2_4} = 2^{-4 k} \langle z^k, K(\cdot,1)  \rangle_{\cH}
    = 2^{ -4 k},
  \end{equation*}
  where we have used Lemma \ref{lem:4_isom} again.

  Hence,
  \begin{equation*}
    \int_{\bS_4} \varphi \, d \mu = \langle M_\varphi 1, g \rangle_{H^2_4}
  \end{equation*}
  for all polynomials $\varphi$, and hence for all $\varphi \in \cA_4$. Since the right-hand side
  obviously extends to a weak-$*$ continuous functional in $\varphi$ on $\cM_4$, we see that
  $\mu$ is $\cA_4$-Henkin.
\end{proof}

\section{The case \texorpdfstring{$d=2$}{d=2}}
\label{sec:d2}

In this section, we will prove Theorem \ref{thm:henkin_measure} in the case $d=2$ and hence
in full generality
by Lemma \ref{lem:extension}.
To this end, we will also embed a reproducing
kernel Hilbert space on $\bD$ into $H^2_2$.
Let
\begin{equation*}
  r: \ol{\bB_2} \to \ol{\bD}, \quad z \mapsto 2 z_1 z_2,
\end{equation*}
and observe that $r$ maps $\bB_2$ onto $\bD$. For $n \in \bN$, let
\begin{equation*}
  a_n = ||r(z)^n||^{-2}_{H^2_2},
\end{equation*}
and consider the reproducing kernel Hilbert space on $\bD$
with reproducing kernel
\begin{equation*}
  K(z,w) = \sum_{n=0}^\infty a_n (z \ol{w})^n.
\end{equation*}

This space turns
out to be the well-known weighted Dirichlet space $\cD_{1/2}$, which is the reproducing
kernel Hilbert space on $\bD$ with reproducing kernel $(1 - z \ol{w})^{-1/2}$. This
explicit description is not strictly necessary for what follows, but it provides
some context for the arguments involving capacity below.

\begin{lem}
  \label{lem:dirichlet_kernel}
  The kernel $K$ satisfies $K(z,w) = (1 - z \ol{w})^{-1/2}$.
\end{lem}

\begin{proof}
  The formula for the norm of monomials in Section \ref{sec:prelim} shows that
  \begin{equation*}
    a_n = ||r(z)^n||^{-2} = 4^{-n} \frac{(2 n)!}{(n!)^2} = (-1)^n \binom{-1/2}{n},
  \end{equation*}
  so that
  \begin{equation*}
    K(z,w) = \sum_{n=0}^\infty (-1)^n \binom{-1/2}{n} (z \ol{w})^n
    = (1 - z \ol{w})^{-1/2}
  \end{equation*}
  by the binomial series.
\end{proof}

The analogue of Lemma \ref{lem:4_isom} in the case $d=2$ is the following result.

\begin{lem}
  \label{lem:2_isom}
  The map
\begin{equation*}
  \cD_{1/2} \to H^2_2, \quad f \mapsto f \circ r,
\end{equation*}
is an isometry. Moreover, $a_n \simeq (n+1)^{-1/2}$ and $||z^n||^2_{\cD_{1/2}} \simeq (n+1)^{1/2}$.
\end{lem}

\begin{proof}
  As in the proof of Lemma \ref{lem:4_isom}, we see that $V$ is an isometry.
  Moreover, Lemma \ref{lem:stirling} shows that
  \begin{equation*}
    ||z^n||^2_{\cD_{1/2}} = ||r(z)^n||^2_{H^2_2} = 2^{2 n} ||(z_1 z_2)^n||^2_{H^2_2} \simeq (n+1)^{1/2}
  \end{equation*}
  for $n \in \bN$.
\end{proof}

The crucial difference to the case $d=4$ is that the
functions in $\cD_{1/2}$ do not all extend to continuous functions on $\ol{\bD}$. This
makes the construction
of the measure $\mu$ of Theorem \ref{thm:henkin_measure} more complicated.

The following lemma provides a measure $\sigma$ on the unit circle which will serve as a replacement
for the Dirac measure $\delta_1$, which was used in the case $d=4$.
It is very likely that this result is well known. Since the measure $\sigma$ is crucial for
the construction of the measure $\mu$,
we explicitly indicate how such a measure on the unit circle can arise.

\begin{lem}
  \label{lem:measure_circle}
  There exists a Borel probability measure $\sigma$ on $\bT$ such that
  \begin{enumerate}[label=\normalfont{(\alph*)}]
    \item the support of $\sigma$ has Lebesgue measure $0$, and
    \item
  the functional
  \begin{equation*}
    \bC[z] \to \bC, \quad p \mapsto \int_{\bT} p \, d \sigma,
  \end{equation*}
  extends to a bounded functional on the space $\cD_{1/2}$.
  \end{enumerate}
\end{lem}

To prove Lemma \ref{lem:measure_circle}, we require the notion of capacity. Background material
on capacity can be found in \cite[Section 2]{EKM+14}.
Let $k(t) = t^{-1/2}$. The \emph{$1/2$-energy} of a Borel probability measure $\nu$ on $\bT$ is defined to be
\begin{equation*}
  I_k(\nu) = \int_{\bT} \int_{\bT} k(|x - y|) d \nu(x) d \nu(y).
\end{equation*}
We say that a compact subset $E \subset \bT$ has
\emph{positive Riesz capacity of degree $1/2$} if there exists a Borel probability measure $\nu$
supported on $E$ with $I_k(\nu) < \infty$.

\begin{proof}[Proof of Lemma \ref{lem:measure_circle}]
  Let $E \subset \bT$ be a compact set with positive Riesz capacity of degree $1/2$, but
  Lebesgue measure $0$. For instance, since $1/2 < \log 2 / \log 3$, the circular middle-third
  Cantor set has this property by \cite[Exercise 2.4.3 (ii)]{EKM+14}. Thus,
  there exists a measure $\sigma$ on $\bT$ whose support is contained in $E$ with $I_k(\sigma) < \infty$.
  Then (a) holds.

  To prove (b),
  for $n \in \bZ$, let
  \begin{equation*}
    \widehat \sigma(n) = \int_{\bT} z^{-n} \, d \sigma(z)
  \end{equation*}
  denote the $n$-th Fourier coefficient of $\sigma$. Since $I_k(\sigma) < \infty$,
  an application of \cite[Exercise 2.4.4]{EKM+14} shows that
  \begin{equation}
    \label{eqn:fourier_decay}
    \sum_{n=0}^\infty \frac{|\widehat \sigma(n)|^2}{ (n+1)^{1/2}} < \infty.
  \end{equation}
  Let now $p$ be a polynomial, say
  \begin{equation*}
    p(z) = \sum_{n=0}^N \alpha_n z^n.
  \end{equation*}
  Then using the Cauchy-Schwarz inequality,
  we see that
  \begin{align*}
    \Big| \int_{\bT} p \, d \sigma  \Big|
    &\le \sum_{n=0}^N |\alpha_n| \, | \widehat \sigma(-n)| \\
    &\le \Big( \sum_{n=0}^N (n+1)^{1/2} |\alpha_n|^2 \Big)^{1/2}
    \Big( \sum_{n=0}^N \frac{|\widehat \sigma(n)|^2}{(n+1)^{1/2}} \Big)^{1/2}.
  \end{align*}
  Lemma \ref{lem:2_isom} shows that the first factor is dominated by $ C ||p||_{\cD_{1/2}}$
  for some constant $C$,
  and the second factor is bounded uniformly in $N$ by \eqref{eqn:fourier_decay}. Thus, (b)
  holds.
\end{proof}

\begin{rem}
  The last paragraph of the proof of \cite[Theorem 2.3.5]{EKM+14} in fact shows that
  the Cantor measure on the circular middle-thirds Cantor set
  has finite $1/2$-energy, thus we can take $\sigma$ to be this measure.
\end{rem}

Let now $\sigma$ be a measure provided by Lemma \ref{lem:measure_circle} and let
$E$ be the support of $\sigma$. Let
\begin{equation*}
  h: \bT \times E \to \bS_2, \quad (\zeta_1,\zeta_2) \mapsto \frac{1}{\sqrt{2}} (\zeta_1, \ol{\zeta_1} \zeta_2),
\end{equation*}
and observe that the range of $h$ is contained in $r^{-1}(E)$. Define $\mu$ to be the pushforward
of $m \times \sigma$ by $h$. We will show that $\mu$ satisfies the conclusion of Theorem \ref{thm:henkin_measure}.

\begin{lem}
  The support of $\mu$ is totally null.
\end{lem}

\begin{proof}
  Let $X = r^{-1}(E)$.
  Since $E$ has Lebesgue measure $0$ by Lemma \ref{lem:measure_circle}, there exists
  by the Rudin-Carleson theorem (i.e. the $d=1$ case of Theorem \ref{thm:peak}) a function
  $f_0 \in A(\bD)$ which peaks on $E$.
  Let $f = f_0 \circ r$. Then $f$ belongs to $A(\bB_d)$ and peaks on $X$,
  so that $X$ is totally null by Theorem \ref{thm:peak}.
  Finally, the support of $\mu$ is contained in $X$, hence it is totally null as well.
\end{proof}

The following lemma finishes the proof of Theorem \ref{thm:henkin_measure}.

\begin{lem}
  The measure $\mu$ is $\cA_2$-Henkin.
\end{lem}

\begin{proof}
  For all $m,n \in \bN$, we have
  \begin{equation*}
    \int_{\bS_2} z_1^m z_2^n \, d \mu = 2^{- (m+n)/2}
    \int_{\bT} \int_E \zeta_1^{m -n} \zeta_2^n \, d m(\zeta_1) \, d \sigma(\zeta_2).
  \end{equation*}
  This quantity is zero unless $m = n$, in which case it equals
  \begin{equation*}
    2^{-n} \int_E \zeta^n d \sigma(\zeta).
  \end{equation*}
  On the other hand,
  Lemma \ref{lem:measure_circle} shows that there exists $f \in \cD_{1/2}$ such that
  \begin{equation*}
    \langle p,f \rangle_{\cD_{1/2}} = \int_{E} p \, d \sigma
  \end{equation*}
  for all polynomials $p$. Let $g = f \circ r$. Then $g$
  belongs to $H^2_d$ by Lemma \ref{lem:2_isom} and it is
  orthogonal to $z_1^n z_2^m$
  unless $n=m$, in which case
  \begin{equation*}
    \langle (z_1 z_2)^n, g \rangle_{H^2_2}
    = 2^{-n} \langle r(z)^n, g \rangle_{H^2_2}
    = 2^{-n} \langle z^n, f \rangle_{\cD_{1/2}}
    = 2^{-n} \int_E \zeta^n \, d\sigma(\zeta).
  \end{equation*}
  Consequently,
  \begin{equation*}
    \int_{\bS_2} \varphi \, d \mu = \langle M_\varphi 1, g  \rangle_{H^2_2}
  \end{equation*}
  for all polynomials $\varphi$, and hence for all $\varphi \in \cA_2$, so that $\mu$ is $\cA_2$-Henkin.
\end{proof}

Theorem \ref{thm:henkin_measure} suggests the following problem, which is deliberately stated somewhat
vaguely.
\begin{prob}
  Find a measure theoretic characterization of $\cA_d$-Henkin measures.
\end{prob}

\bibliographystyle{amsplain}
\bibliography{../../Dropbox/Literature/jabref_database}

\end{document}